\documentclass[11pt]{amsart}

\usepackage{amsmath,amssymb,fullpage}
\usepackage{algorithm, algorithmic}
\usepackage{arydshln,url}

\newtheorem{thm}{Theorem}
\newtheorem{lemma}{Lemma}
\newtheorem{coro}{Corollary}
\newtheorem{definition}{Definition}
\newtheorem{remark}{Remark}

\DeclareMathOperator{\rank}{rank}

\usepackage{tcolorbox}
\usepackage[framemethod=TikZ]{mdframed}
\mdfdefinestyle{style1}{
	linecolor=gray,
	outerlinewidth=1.5pt,
	roundcorner=0pt,
	innertopmargin=\baselineskip,
	innerbottommargin=\baselineskip,
	innerrightmargin=20pt,
	innerleftmargin=20pt,
	backgroundcolor=gray!10!white}

 \title{A Distributed and Incremental SVD Algorithm for Agglomerative
   Data Analysis on Large Networks}
   
 \author{M.~A.~Iwen,
   B.~W.~Ong}

\usepackage{pgf,tikz}
\usepackage{pgfplots}
\usetikzlibrary{shapes.geometric,arrows}
\usetikzlibrary{positioning}
\pgfplotsset{compat=newest}

\begin{document}
\maketitle

\begin{abstract}
  In this paper it is shown that the SVD of a matrix can be constructed
  efficiently in a hierarchical approach.  The proposed algorithm is proven to
  recover the singular values and left singular vectors of the input matrix $A$ if its
  rank is known.  Further, the hierarchical
  algorithm can be used to recover the $d$ largest singular values and
  left singular vectors with bounded error.  It is also shown that the
  proposed method is stable with respect to roundoff errors or
  corruption of the original matrix entries.  Numerical experiments
  validate the proposed algorithms and parallel cost analysis.
\end{abstract}


%

\section{Introduction}

The singular value decomposition (SVD) of a matrix,
\begin{align}
  A = U \Sigma V^*,
\end{align}
has applications in many areas including principal component analysis
\cite{jolliffe2002principal}, the solution to homogeneous linear
equations, and low-rank matrix approximations.  If $A$ is a complex
matrix of size $D\times N$, then the factor $U$ is a unitary matrix of
size $D\times D$ whose first nonzero entry in each column is a
positive real number
\footnote{This last condition on $U$ guarantees that the SVD of $A \in
  \mathbb{C}^{N \times N}$ will be unique whenever $AA^*$ has no
  repeated eigenvalues.},  $\Sigma$ is a rectangular matrix of size
$D\times N$ with non-negative real numbers (known as singular values)
ordered from largest to smallest down its diagonal, and $V^*$ (the
conjugate transpose of $V$) is also a unitary matrix of size $N\times
N$.  If the matrix $A$ is of rank $d < \min(D,N)$, then a reduced SVD
representation is possible:
\begin{align}
  A = \hat{U} \hat{\Sigma} \hat{V}^*,
\end{align}
where $\hat{\Sigma}$ is a $d\times d$ diagonal matrix with positive
singular values, $\hat{U}$ is an $D\times d$ matrix with orthonormal
columns, and $\hat{V}$ is a $d\times N$ matrix with orthonormal
columns.

The SVD of $A$ is typically computed in three stages: a bidiagonal
reduction step, computation of the singular values, and then
computation of the singular vectors.  The bidiagonal reduction step is
computationally intensive, and is often targeted for parallelization.
A serial approach to the bidiagonal reduction is the Golub--Kahan
bidiagonalization algorithm \cite{doi:10.1137/0702016}, which reduces
the matrix A to an upper-bidiagonal matrix by applying a series of
Householder reflections alternately, applied from the left and right.
Low-level parallelism is possible by distributing matrix-vector
multiplies, for example by using the cluster computing framework Spark
\cite{Zaharia:2012:RDD:2228298.2228301}.  Using this form of low-level
parallelism for the SVD has been implemented in the Spark project
MLlib \cite{DBLP:journals/corr/MengBYSVLFTAOXX15}, and Magma
\cite{1742-6596-180-1-012037}, which develops its own framework to
leverage GPU accelerators and hybrid manycore systems.  Alternatively,
parallelization is possible on an algorithmic level.  For example, it
is possible to apply independent reflections simultaneously; the
bidiagonalization has been mapped to graphical processing (GPU) units
\cite{Liu:2010:GPH:1851476.1851512} and to a distributed cluster
\cite{ls2009}.  Load balancing is an issue for such parallel
algorithms, however, because the number of off-diagonal columns (or
rows) to eliminate get successively smaller.  More recently, two-stage
approaches have been proposed and utilized in high-performance
implementations for the bidiagonal reduction
\cite{Ltaief:2013:HBR:2450153.2450154,Haidar:2013:IPS:2503210.2503292}.
The first stage reduces the original matrix to a banded matrix, the
second stage subsequently reduces the banded matrix to the desired
upper-bidiagonal matrix.  Further, these algorithms can be optimized
to hide latency and cache misses
\cite{Haidar:2013:IPS:2503210.2503292}.  The SVD of a bidiagonal
matrix can be computed in parallel using divide and conquer mechanisms
based on rank one tearings \cite{ doi:10.1137/S089547989120195X}.
Parallelization is also possible if one uses a probabilistic approach
to approximating the SVD \cite{halko2011finding}.

In this paper, we are concerned with finding the SVD of highly
rectangular matrices, $N\gg D$.  In many applications where such
problems are posed, one typically cares about the singular values, the
left singular vectors, or their product. For example, this work was
motivated by the SVDs required in Geometric Multi-Resolution Analysis
(GMRA) \cite{allard2011multiscale}; the higher-order singular value
decomposition (HOSVD) \cite{de2000multilinear} of a tensor requires
the computation of $n$ SVDs of highly rectangular matrices, where $n$
is the number of tensor modes.  Similarly, tensor train factorization
algorithms \cite{oseledets2011tensor} for tensors require the
computation of many highly rectangular SVDs. Indeed, the SVDs of
distributed and highly rectangular matrices of data appear in many
big-data machine learning applications.

To find the SVD of highly rectangular matrices, many methods have
focused on randomized techniques \cite{Ma2015statistical}.  Another
approach is to compute the eigenvalue decomposition of the Gram
matrix, $AA^*$ \cite{2015arXiv151006689A}.  Although computing the
Gram matrix in parallel is straightforward using the block inner
product, a downside to this approach is a loss of numerical precision,
and the general availability of the entire matrix $A$, to which one
may not have easy access (i.e., computation of the Gram matrix,
$AA^*$, is not easily achieved in an incremental and distributed
setting).  One can instead compute the SVD of the matrix incrementally
-- such methods have previously been developed to efficiently analyze
data sets whose data is only measured incrementally, and have been
extensively studied in the machine-learning community to identify
low-dimensional subspaces \cite{139256, doi:10.1137/0613061,
  1992ITSP...40..571D, Li20041509,Skocaj200827,6713992}.  In early
work, algorithms were developed to update an SVD decomposition when a
row or column is appended \cite{MR509670} using rank one updates
\cite{MR508586,MR0329227}. These scheme were potentially unstable, but
were later stablized \cite{Gu94,doi:10.1137/S0895479893251472} and a
fast version proposed \cite{MR2214744}.  An SVD-updating algorithm for
low-rank approximations was presented by Zha in Simon in 1990
\cite{MR1718703}.  If rank $d$ approximation of $A$ is known,
i.e. $A=\hat{U}\hat{\Sigma}\hat{V}^*$, rank $d$ approximation of
$[A,B]$ can be constructed by taking
\begin{enumerate}
\item The QR decomposition of $(I - \hat{U}\hat{U}^*)B = QR$, 
\item finding the rank $d$ SVD of
  \begin{align*}
    \left[
      \begin{tabular}{cc}
        $\hat{\Sigma}$ & $\hat{U}^*$B \\
        0 & R
      \end{tabular}
      \right]
    = \tilde{U}\tilde{\Sigma}\tilde{V}^*, {\rm and ~then}
  \end{align*}
\item forming the best rank $d$ approximation:
  \begin{align*}
    \left([\hat{U}, Q]\tilde{U}\right) \tilde{\Sigma}
    \left(
    \left[
      \begin{tabular}{cc}
        $\hat{V}$ & 0 \\
        0 & I
      \end{tabular}
      \right]
    \tilde{V}
    \right) ^*.
  \end{align*}
\end{enumerate}
This algorithm was adapted for eigen decompositions \cite{Kwok03} and
later utilized to construct incremental PCA algorithms \cite{1658299}.
Another block-incremental approach for estimating the dominant
singular values and vectors of a highly rectangular matrix uses a QR
factorization of blocks from the input matrix, which can be done
efficiently in parallel \cite{MR2908601}.  In fact, the QR
decomposition can be computed using a communication-avoiding QR (CAQR)
factorization \cite{MR2890264}, which utilizes a tree-reduction
approach.  Our approach is similar in spirit to the CAQR factorization
above \cite{MR2890264}, but differs in that we employ a block
decomposition approach that utilizes a partial SVD rather than a full
QR factorization.  This is advantageous if the application only
requires the singular values and/or left singular vectors of the input matrix $A$, as in tensor
factorization \cite{de2000multilinear,oseledets2011tensor} and GMRA
applications \cite{allard2011multiscale}.

The remainder of the paper is laid out as follows: In
Section~\ref{sec:hierarchical_svd}, we motivate incremental approaches
to constructing the SVD before introducing the hierarchical algorithm.
Theoretical justifications are given to show that the algorithm
exactly recovers the singular values and left singular vectors if the
rank of the matrix $A$ is known.  An error analysis is also used to
show that the hierarchical algorithm can be used to recover the $d$
largest singular values and left singular vectors with bounded error,
and that the algorithm is stable with respect to roundoff errors or
corruption of the original matrix entries.  In
Section~\ref{sec:numerical_results}, numerical experiments validate
the proposed algorithms and parallel cost analysis.

\section{An Incremental (hierarchical) SVD Approach}
\label{sec:hierarchical_svd}

The overall idea behind the proposed approach is relatively simple.
We require a {\it distributed} and {\it incremental} approach for
computing the singular values and left singular vectors of all data
stored across a large distributed network.  This can be achieved, for
example, by occasionally combining a previously computed
partial SVD representation of each node's past data with a new partial SVD of
its more recent data.  As a result of this approach each
separate network node will always contain a fairly accurate approximation
of its cumulative data over time.  Of course, these separate nodes'
partial SVDs must then be merged together in order to understand the network
data as a whole.  Toward this end, partial SVD approximations of
neighboring nodes can also be combined together hierarchically in order to eventually compute a
global partial SVD of the data stored across the entire network.

Note that the accuracy of the entire approach described above will be determined by
the accuracy of the (hierarchical) partial SVD merging technique,
which is ultimately what leads to the proposed method being both
incremental and distributed.  Theoretical analysis of this partial SVD
merging technique is the primary purpose of this section.  In
particular, we prove that the proposed partial SVD merging scheme is
numerically robust to both data and roundoff errors.  In addition, the merging scheme is also shown to be accurate
even when the rank of the overall data matrix $A$ is underestimated
and/or purposefully reduced.

\subsection{Mathematical Preliminaries}
Let $A \in \mathbb{C}^{D \times N}$ be a highly rectangular matrix,
with $N \gg D$.  Further, let $A^i \in \mathbb{C}^{D\times N_i}$ with
$i=1,2,\ldots,M$, denote the block decomposition of $A$, i.e.,
$A=\left[A^1|A^2|\cdots|A^M\right]$.

\begin{definition}
For any matrix $A \in \mathbb{C}^{D \times N}$, $(A)_d\in \mathbb{C}^{D
  \times N}$ is an optimal rank $d$ approximation to $A$ with respect
to Frobenius norm $\| \cdot \|_{\rm F}$ if
\begin{align*}
  \inf_{B\in \mathbb{C}^{D\times N}} \|B-A\|_F = \|(A)_d - A\|_F,
  \text{ subject to } \rank{(B)} \le d.
\end{align*}
\end{definition}
If $A$ has the SVD decomposition $A = U\Sigma V^*$, then $(A)_d
= \sum_{i=1}^d u_i \sigma_i v_i^*$, where $u_i$ and $v_i$ are singular
vectors that comprise $U$ and $V$ respectively, and $\sigma_i$ are
the singular values.

\begin{remark}
  The Frobenius norm can be computed using $ \|A\|^2_F = \sum_i^D
  \sigma_i^2$ Consequently, $ \|(A)_d - A \|^2_F = \sum_{d+1}^D
  \sigma_i^2$.
\end{remark}

This following lemma, Lemma~\ref{lemma:merge} proves that partial SVDs
of blocks of our original data matrix, $A \in \mathbb{C}^{D\times N}$,
can be combined block-wise into a new reduced matrix $B$ which has the
same singular values and left singular vectors as the original $A$.
This basic lemma can be considered as the simplest merging method for
either constructing an incremental SVD approach (different blocks of
$A$ have their partial SVDs computed at different times, which are
subsequently merged into $B$), a distributed SVD approach (different
nodes of a network compute partial SVDs of different blocks of $A$
separately, and then send them to a single master node for combination
into $B$), or both.

\begin{lemma}
  \label{lemma:merge}
  Suppose that $A \in \mathbb{C}^{D\times N}$ has rank
  $d\in\{1,\ldots,D\}$, and let $A^i \in \mathbb{C}^{D\times N_i},
  i=1,2,\ldots,M$ be the block decomposition of $A$, i.e.,
  $A=\left[A^1|A^2|\cdots|A^M\right]$.  Since $A^i$ has rank at most $d$,
  each block has a reduced SVD representation,
  \begin{align*}
    A^i = \sum_{j=1}^d u_j^i \sigma_j^i (v_j^i)^* = \hat{U}^i
    \hat{\Sigma}^i \hat{V}^{i*} , \quad i = 1,2,\ldots,M.
  \end{align*}
  Let $B := \left[\hat{U}^1 \hat{\Sigma}^1 | \hat{U}^2 \hat{\Sigma}^2
    |\cdots|\hat{U}^M\hat{\Sigma}^M \right]$.  If $A$ has the reduced
  SVD decomposition, $A = \hat{U}\hat{\Sigma}\hat{V}^*$, and $B$ has
  the reduced SVD decomposition, $B =
  \hat{U}'\hat{\Sigma}'\hat{V'}^*$, then $\hat{\Sigma} =
  \hat{\Sigma}'$, and $\hat{U} = \hat{U}'W$, where $W$ is a unitary
  block diagonal matrix.  If none of the nonzero singular values are
  repeated then $\hat{U} = \hat{U}'$ (i.e., $W$ is the identity when
  all the nonzero singular values of $A$ are unique).
\end{lemma}

\begin{proof}
  The singular values of $A$ are the (non-negative) square root of the
  eigenvalues of $A A^*$.  Using the block definition of $A$,
  \begin{align*}
    A A^* &= \sum_{i=1}^M A^i (A^i)^* 
    = \sum_{i=1}^M  \hat{U}^i \hat{\Sigma}^i (\hat{V}^i)^*   (\hat{V}^i) (\hat{\Sigma}^i)^* (\hat{U}^i)^* 
    = \sum_{i=1}^M  \hat{U}^i \hat{\Sigma}^i (\hat{\Sigma}^i)^* (U^i)^*
  \end{align*}
  Similarly, the singular values of $B$ are the (non-negative) square
  root of the eigenvalues of $B B^*$.
  \begin{align*}
    B B^* &= \sum_{i=1}^M (\hat{U}^i \hat{\Sigma}^i) (\hat{U}^i \hat{\Sigma}^i)^* 
    = \sum_{i=1}^M  \hat{U}^i \hat{\Sigma}^i (\hat{\Sigma}^i)^* (\hat{U}^i)^*
  \end{align*}
  Since $AA^* = BB^*$, the singular values of $B$ must be the same as
  the singular values of $A$.  Similarly, the left singular vectors of
  both $A$ and $B$ will be eigenvectors of $A A^*$ and $B B^*$,
  respectively.  Since $AA^* = BB^*$ the eigenspaces associated with each (possibly repeated) eigenvalue will also be
  identical so that $\hat{U} = \hat{U}'W$.  The block diagonal unitary matrix $W$ (with one unitary $h \times h$ block for each eigenvalue that is repeated $h$-times)
  allows for singular vectors associated with repeated singular
  values to be rotated in the matrix representation $\hat{U}$.
\end{proof}

We now propose and analyze a more useful SVD approach which takes the ideas present in Lemma~\ref{lemma:merge} to their logical conclusion.

\subsection{An Incremental (Hierarchical) SVD Algorithm}

The idea is to leverage the result in Lemma~\ref{lemma:merge} by
computing (in parallel) the SVD of the blocks of $A$, concatenating
the scaled left singular vectors of the blocks to form a proxy matrix
$B$, and then finally recovering the singular values and left singular
vectors of the original matrix $A$ by finding the SVD of the proxy
matrix.  A visualization of these steps are shown in
Figure~\ref{fig:cartoon_eg}.
\begin{figure}[htbp]
  \centering
  \includegraphics{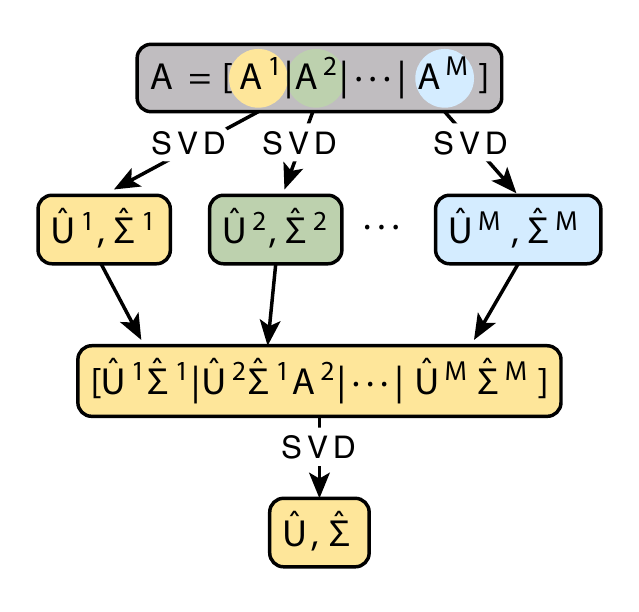}
  \caption{Flowchart for a simple (one-level) distributed parallel SVD
    algorithm.  The different colors represent different processors
    completing operations in parallel.}
  \label{fig:cartoon_eg}
\end{figure}
Provided the proxy matrix is not very large, the computational and
memory bottleneck of this algorithm is in the simultaneous SVD
computation of the blocks $A^i$.  If the proxy matrix is sufficiently
large that the computational/memory overhead is significant, a
multi-level hierarchical generalization is possible through repeated
application of Lemma~\ref{lemma:merge}.  Specifically, one could
generate multiple proxy matrices by concatenating subsets of scaled
left singular vectors obtained from the SVD of blocks of A, find the
SVD of the proxy matrices and concatenate those singular vectors to
form a new proxy matrix, and then finally recover the singular values
and left singular vectors of the original matrix $A$ by finding the
SVD of the proxy matrix.  A visualization of this generalization is
shown in Figure~\ref{fig:cartoon_eg2} for a two-level parallel
decomposition.  A general $q$-level algorithm is described in
Algorithm~\ref{algorithm:parallel_svd}.
\begin{figure}[htbp]
  \centering
  \includegraphics{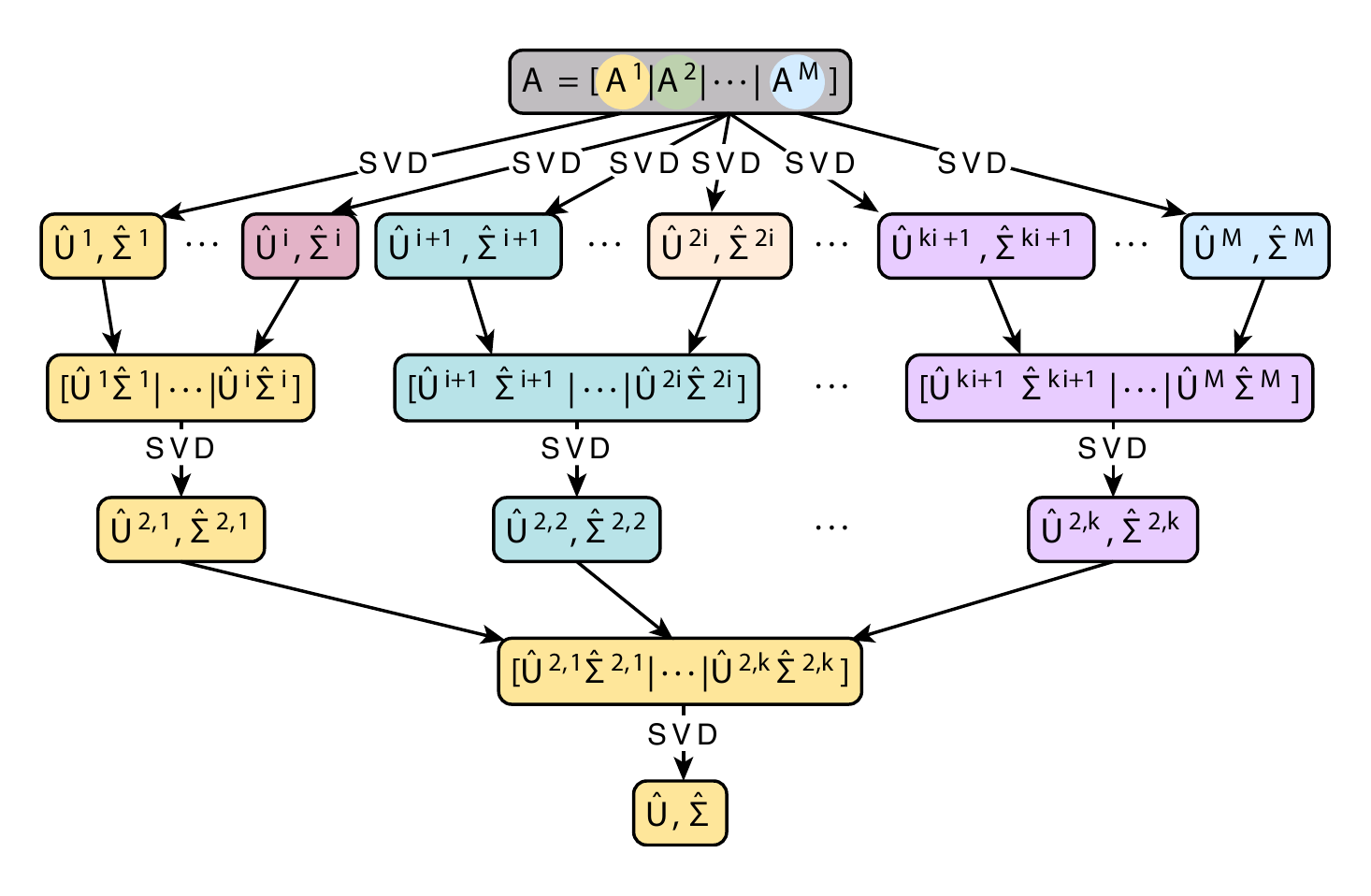}
  \caption{Flowchart for a two-level hierarchical parallel SVD
    algorithm. The different colors represent different processors
    completing operations in parallel.}
  \label{fig:cartoon_eg2}
\end{figure}
\begin{algorithm}
\renewcommand{\algorithmicrequire}{\textbf{Input:}}
\renewcommand{\algorithmicensure}{\textbf{Output:}}
\caption{A $q$-level, distributed SVD Algorithm for Highly Rectangular
  $A \in \mathbb{C}^{D \times N}$, ~ $N \gg D$. }
\label{alg:parallelSVD}
\begin{algorithmic}[1]
  \REQUIRE 
  $q$ (\# levels), \\
  $n$ (\# local SVDs to concatenate at each level), \\
  $d \in \{ 1, \dots, D \}$ (intrinsic dimension), \\
  $A^{1,i} := A^i \in \mathbb{C}^{D\times N_i}$ for $i=1,2,\ldots,M$ (block decomposition of $A$; algorithm assumes $ M
  = n^q$ -- generalization is trivial)\\
  
  \ENSURE $U' \in \mathbb{C}^{D \times d} \approx$ the first $d$
  columns of $U$, and $\Sigma' \in \mathbb{R}^{d \times d} \approx
  \left(\Sigma \right)_d$.

  \FOR {$p = 1, \dots, q$ }

  \STATE Compute (in parallel) the SVDs of $A^{p,i} = U^{p,i} \Sigma^{p,i}
  \left( V^{p,i} \right)^*$ for $i=1,2,\ldots,M / n^{(p-1)}$, unless the $U^{p,i} \Sigma^{p,i}$ are already available from a previous run.
  
  \STATE Set $\displaystyle A^{p+1,i} := \left[ \left( U^{p,(i-1)n+1} \Sigma^{p,(i-1)n+1}
    \right)_d ~\Big| \cdots \Big|~ \left( U^{p,in} \Sigma^{p,in} \right)_d \right]$
  for $i=1,2,\ldots,M / n^p$.
  
  \ENDFOR
  
  \STATE Compute the SVD of $A^{q+1,1}$

  \STATE Set $U' :=$ the first $d$ columns of $U^{q+1,1}$, and $\Sigma'
  := \left( \Sigma^{q+1,1} \right)_d$.
    \end{algorithmic}
\label{algorithm:parallel_svd}
\end{algorithm}

\begin{remark}
 The right singular vectors can be computed (in parallel) if desired,
 once the left singular vectors and singular values are known.  The
 master process broadcasts the left singular vectors and singular
 values to each process containing block $A^i$.  Then columns of the
 right singular vectors can be constructed by computing
 $\frac{1}{\sigma_j}(A^i)^* u_j$, were $A^i$ is the block of $A$
 residing on process $i$, and $(\sigma_j,u_j)$ is the $j^{th}$
 singular value and left singular vector respectively.
\end{remark}

\subsection{Theoretical Justification}

In this section we will introduce some additional notation for the
sake of convenience.  For any matrix $A \in \mathbb{C}^{D \times N}$
with SVD $A = U \Sigma V^*$, we will let $\overline{A} := U \Sigma = A
V \in \mathbb{C}^{D \times N}$.  It is important to note that
$\overline{A}$ is not necessarily uniquely determined by $A$, for
example, if $A$ is rank deficient and/or has repeated singular values.
In these types of cases many pairs of unitary $U$ and $V$ may appear
in a valid SVD of $A$.  In what follows, one can consider
$\overline{A}$ to be $A V$ {\it for any such valid unitary matrix
  $V$}.  Similarly, one can always consider statements of the form
$\overline{A} = \overline{B}$ as meaning that $A$ and $B$ are
equivalent up to multiplication by a unitary matrix on the right.
This inherent ambiguity does not effect the results below in a
meaningful way.  Given this notation Lemma~\ref{lemma:merge} can be
rephrased as follows:

\begin{coro}
Suppose that $A \in \mathbb{C}^{D\times N}$ has rank
  $d\in\{1,\ldots,D\}$, and let $A^i \in \mathbb{C}^{D\times N_i},
  i=1,2,\ldots,M$ be the block decomposition of $A$, i.e.,
  $A=\left[A^1|A^2|\cdots|A^M\right]$.  Since $A^i$ has rank at most $d$ for all $i=1,2,\ldots,M$,
we have that 
$$\overline{(A)_d} ~=~ \overline{A} ~=~ \overline{\left[ \overline{A^1} ~\big|~ \overline{A^2} ~\big|~ \cdots ~\big|~ \overline{A^M}\right]} ~=~ \overline{\left[ \overline{(A^1)_d} ~\big|~ \overline{(A^2)_d} ~\big|~ \cdots ~\big|~ \overline{(A^M)_d}\right]}.$$
\label{coro:merge}
\end{coro}

\begin{proof}
To begin, we apply Lemma~\ref{lemma:merge} with $A=\left[A^1|A^2|\cdots|A^M\right]$ and\\ $B = \left[ \overline{A^1} ~\big|~ \overline{A^2} ~\big|~ \cdots ~\big|~ \overline{A^M}\right]$.  As a result we learn that if $A$ has the SVD $A = U \Sigma V^*$, then $B$ also has a valid SVD of the form $B = (UW) \Sigma {V'}^*$, where $W$ is a block diagonal unitary matrix with one unitary $h \times h$ block associated with each singular value that is repeated $h$-times (this allows for left singular vectors associated with repeated singular values to be rotated in the unitary matrix $U$).  Note that the special structure of $W$ with respect to $\Sigma$ allows them to commute.  Hence, we have that 
$$B = (UW) \Sigma {V'}^* ~=~ U \Sigma W {V'}^* ~=~ U \Sigma (V' W^*)^*.$$
As a result we can see that $A = B (V' W^* V^*)$ which, in turn, guarantees that $\overline{A} = \overline{B}$.  Hence, 
\begin{equation}
\overline{\left[A^1|A^2|\cdots|A^M\right]} = \overline{A} = \overline{B} = \overline{\left[ \overline{A^1} ~\big|~ \overline{A^2} ~\big|~ \cdots ~\big|~ \overline{A^M}\right]}.
\label{equ:proofcoro:merge}
\end{equation}

To finish, recall that $A$ has rank $d$.  Thus, all it's blocks, $A^i$ for $i=1,2,\ldots,M$, are also of rank at most $d$ (to see why, consider, e.g., the reduction of $A=\left[A^1|A^2|\cdots|A^M\right]$ to row echelon form).  As a result, both $(i)$ $A = (A)_d$, and $(ii)$ $A^i = (A^i)_d$ for all $i=1,2,\ldots,M$, are true.  Substituting these equalities into \eqref{equ:proofcoro:merge} finishes the proof.
\end{proof}

We can now prove that Algorithm~\ref{alg:parallelSVD} is guaranteed to recover $\overline{A}$ when the rank of $A$ is known.  The proof follows by inductively applying Corollary~\ref{coro:merge}.

\begin{thm}
Suppose that $A \in \mathbb{C}^{D\times N}$ has rank $d\in\{1,\ldots,D\}$.  Then, Algorithm~\ref{alg:parallelSVD} is guaranteed to recover an $A^{q+1,1} \in \mathbb{C}^{D\times N}$ such that $\overline{A^{q+1,1}} = \overline{A}$.
\label{thm:FullRankRecovery}
\end{thm}

\begin{proof}
We prove the theorem by induction on the level $p$.  To establish the base case we note that 
$$\overline{A} ~=~ \overline{\left[ \overline{(A^{1,1})_d} ~\big|~ \overline{(A^{1,2})_d} ~\big|~ \cdots ~\big|~ \overline{(A^{1,M})_d}\right]} ~=~  \overline{\left[ \overline{A^{1,1}} ~\big|~ \overline{A^{1,2}} ~\big|~ \cdots ~\big|~ \overline{A^{1,M}}\right]}$$
holds by Corollary~\ref{coro:merge}.  Now, for the purpose of induction, suppose that 
$$\overline{A} ~=~ \overline{\left[ \overline{(A^{p,1})_d} ~\big|~ \overline{(A^{p,2})_d} ~\big|~ \cdots ~\big|~ \overline{(A^{p,M/n^{(p-1)}})_d}\right]} ~=~ \overline{\left[ \overline{A^{p,1}} ~\big|~ \overline{A^{p,2}} ~\big|~ \cdots ~\big|~ \overline{A^{p,M/n^{(p-1)}}}\right]}$$
holds for some some $p \in \{ 1, \dots, q\}$.  Then, we can use the induction hypothesis and repartition the blocks of $\overline{A}$ to see that
\begin{align}
\overline{A} ~&=~ \overline{\left[ \overline{(A^{p,1})_d} ~\big|~ \overline{(A^{p,2})_d} ~\big|~ \cdots ~\big|~ \overline{(A^{p,M/n^{(p-1)}})_d}\right]} \nonumber \\
~&=~ \overline{\left[ 
\cdots ~\Big|~ \big[ ( \overline{A^{p,(i-1)n+1} } )_d  \dots  ( \overline{A^{p,in}} )_d \big] ~\Big|~ \cdots  \right]}, ~i = 1, \dots, M/n^p \nonumber \\
~&=~  \overline{\left[ A^{p+1,1} ~\big|~ A^{p+1,2} ~\big|~ \cdots ~\big|~ A^{p+1,M/n^{p}} \right]},
\label{equ:Inducts}
\end{align}
where we have utilized the definition of $A^{p+1,i}$ from line 3 of Algorithm~\ref{alg:parallelSVD} to get \eqref{equ:Inducts}.  Applying Corollary~\ref{coro:merge} to the matrix in \eqref{equ:Inducts} now yields
$$\overline{A} ~=~ \overline{\left[ \overline{A^{p+1,1}} ~\big|~ \overline{A^{p+1,2}} ~\big|~ \cdots ~\big|~ \overline{A^{p+1,M/n^{p}}} \right]}.$$
Finally, we finish by noting that each $\overline{A^{p+1,i}}$ will have rank at most $d$ since $\overline{A}$ is of rank $d$.  Hence, we will also have $\overline{A} ~=~ \overline{\left[ \overline{(A^{p+1,1})_d} ~\big|~ \overline{(A^{p+1,2})_d} ~\big|~ \cdots ~\big|~ \overline{(A^{p+1,M/n^{p}})_d} \right]}$, finishing the proof.
\end{proof}

Our next objective is to understand the accuracy of Algorithm~\ref{alg:parallelSVD} when it is called with a value of $d$ that is less than rank of $A$.  To begin, we need to better understand how accurate blockwise low-rank approximations of a given matrix $A$ are.  The following lemma provides an answer.

\begin{lemma}
  Suppose $A^i \in \mathbb{C}^{D\times N_i}, i=1,2,\ldots,M$.
  Further, suppose matrix $A$ has block components
  $A=\left[A^1|A^2|\cdots|A^M\right]$, and $B$ has block components
  $B=\left[(A^1)_d | (A^2)_d |\cdots|(A^M)_d \right]$.  Then, $\|
  (B)_d - A \|_{\rm F} \leq \| (B)_d - B \|_{\rm F} +  \| B - A \|_{\rm F} \leq 3 \| (A)_d - A \|_{\rm F}$ holds for all $d \in \{ 1, \dots, D\}$.
  \label{Lemma:LowrankMerge}
\end{lemma}

\begin{proof}
We have that 
\begin{align*}
\| (B)_d - A \|_{\rm F} &\leq \| (B)_d - B \|_{\rm F} +  \| B - A \|_{\rm F}\\ &\leq \| (A)_d - B \|_{\rm F} +  \| B - A \|_{\rm F}\\
&\leq \| (A)_d - A \|_{\rm F} + 2 \| B - A \|_{\rm F}.
\end{align*}
Now letting $(A)^i_d \in \mathbb{C}^{D\times N_i}, i=1,2,\ldots,M$ denote the $i^{\rm th}$ block of $(A)_d$, we can see that
\begin{align*}
\| B - A \|^2_{\rm F} &= \sum^M_{i=1} \| (A^i)_d - A^i \|^2_{\rm F} \\ &\leq \sum^M_{i=1} \| (A)^i_d - A^i \|^2_{\rm F}\\ &= \| (A)_d - A \|^2_{\rm F} .
\end{align*}
Combining these two estimates now proves the desired result.
\end{proof}

We can now use Lemma~\ref{Lemma:LowrankMerge} to prove a theorem that
 will help us to bound the error produced by
Algorithm~\ref{algorithm:parallel_svd} for $q = 1$ when $d$ is chosen to be less
than the rank of $A$.  It improves over Lemma~\ref{Lemma:LowrankMerge}
(in our setting) by not implicitly assuming to have access to any
information regarding the right singular vectors of the blocks of $A$.
It also demonstrates that the proposed method is stable with respect
to additive errors by allowing (e.g., roundoff) errors, represented by
$\Psi$, to corrupt the original matrix entries.  Note that
Theorem~\ref{thm:LowrankMerge} is a strict generalization of
Corollary~\ref{coro:merge}.  Corollary~\ref{coro:merge} is recovered
from it when $\Psi$ is chosen to be the zero matrix, and $d$ is chosen
to be the rank of $A$.

\begin{thm}
Suppose that $A \in \mathbb{C}^{D\times N}$ has block components $A^i \in \mathbb{C}^{D\times N_i}, i=1,2,\ldots,M$, so that $A=\left[A^1|A^2|\cdots|A^M\right]$.
Let $B=\left[\overline{(A^1)_d} ~\big|~ \overline{(A^2)_d} ~\big|~ \cdots~\big|~ \overline{(A^M)_d} \right]$, $\Psi \in \mathbb{C}^{D\times N}$, and $B' = B + \Psi$.  Then, there exists a unitary matrix $W$ such that
$$\left\| \overline{\left( B' \right)_d} - AW \right\|_{\rm F} \leq 3\sqrt{2} \| (A)_d - A \|_{\rm F} + \left(1+ \sqrt{2} \right) \| \Psi \|_{\rm F}$$
holds for all $d \in \{ 1, \dots, D\}$.
\label{thm:LowrankMerge}
\end{thm}

\begin{proof}
Let $A' = \left[\overline{A^1} ~\big|~ \overline{A^2} ~\big|~\cdots ~\big|~ \overline{A^M}\right]$.  Note that $\overline{A'} = \overline{A}$ by Corollary~\ref{coro:merge}.  Thus, there exists a unitary matrix $W''$ such that $A' = \overline{A} W''$.  Using this fact in combination with the unitary invariance of the Frobenius norm, one can now see that
$$\left\| \left( B' \right)_d - A' \right\|_{\rm F} ~=~ \left\| \left( B' \right)_d - \overline{A} W'' \right\|_{\rm F} ~=~\left\| \overline{\left( B' \right)_d} - \overline{A}W' \right\|_{\rm F} = \left\| \overline{\left( B' \right)_d} - AW \right\|_{\rm F}$$
for some unitary matrixes $W'$ and $W$.  Hence, it suffices to bound $\left\| \left( B' \right)_d - A' \right\|_{\rm F}$.  

Proceeding with this goal in mind we can see that
\begin{align*}
\left\| \left( B' \right)_d - A' \right\|_{\rm F} ~&\leq~ \left\| \left( B' \right)_d - B' \right\|_{\rm F} + \left\| B'  - B \right\|_{\rm F} + \left\| B - A' \right\|_{\rm F}\\
~&=~ \sqrt{\sum^D_{j = d+1} \sigma^2_j(B + \Psi)} ~+~ \left\| \Psi \right\|_{\rm F} +  \left\| B - A' \right\|_{\rm F}\\
~&=~ \sqrt{\sum^{\left\lceil \frac{D-d}{2} \right\rceil}_{j = 1} \sigma^2_{d+2j-1}(B + \Psi) + \sigma^2_{d+2j}(B + \Psi)} ~+~ \left\| \Psi \right\|_{\rm F} +  \left\| B - A' \right\|_{\rm F}\\
~&\leq~\sqrt{\sum^{\left\lceil \frac{D-d}{2} \right\rceil}_{j = 1} \left(\sigma_{d+j}(B) + \sigma_{j}(\Psi) \right)^2 + \left( \sigma_{d+j}(B) + \sigma_{j+1}(\Psi)\right)^2} ~+~ \left\| \Psi \right\|_{\rm F} +  \left\| B - A' \right\|_{\rm F}
\end{align*}
where the last inequality results from an application of Weyl's inequality to the first term \cite{horn1991topics}.  Utilizing the triangle inequality on the first term now implies that 
\begin{align*}
\left\| \left( B' \right)_d - A' \right\|_{\rm F} ~&\leq~ \sqrt{\sum^D_{j = d+1} 2 \sigma^2_j(B) } ~+~ \sqrt{\sum^D_{j = 1} 2 \sigma^2_j(\Psi) }  + \left\| \Psi \right\|_{\rm F} +  \left\| B - A' \right\|_{\rm F}\\
&\leq~ \sqrt{2} \left( \| (B)_d - B \|_{\rm F} +  \| B - A' \|_{\rm F} \right)+ \left(1+ \sqrt{2} \right) \| \Psi \|_{\rm F}.
\end{align*}
Applying Lemma~\ref{Lemma:LowrankMerge} to bound the first two terms now concludes the proof after noting that $\| (A')_d - A' \|_{\rm F} =  \| (A)_d - A \|_{\rm F}$.
\end{proof}

This final theorem bounds the total error of the general $q$-level hierarchical 
Algorithm~\ref{algorithm:parallel_svd} with respect to the true matrix
$A$ (up to multiplication by a unitary matrix on the right).  The structure of
its proof is similar to that of Theorem \ref{thm:FullRankRecovery}.

\begin{thm}
  Let $A \in \mathbb{C}^{D\times N}$ and $q \geq 1$.  Then,
  Algorithm~\ref{alg:parallelSVD} is guaranteed to recover an
  $A^{q+1,1} \in \mathbb{C}^{D\times N}$ such that
  $\overline{\left(A^{q+1,1} \right)_d} = A W + \Psi$, where $W$ is a
  unitary matrix, and $\| \Psi \|_{\rm F} \leq \left( \left( 1 +
  \sqrt{2} \right)^{q+1} - 1 \right) \| (A)_d - A \|_{\rm F}$.
  \label{thm:LowRankRecovery}
\end{thm}

\begin{proof}
  Within the confines of this proof we will always refer to the
  approximate matrix $A^{p+1,i}$ from line 3 of
  Algorithm~\ref{alg:parallelSVD} as $$B^{p+1,i} := \left[
    \overline{\left( B^{p,(i-1)n+1} \right)_d} ~\Big| \cdots \Big|~
    \overline{\left( B^{p,in} \right)_d} \right],$$ for $p = 1,\dots,
  q$, and $i = 1, \dots, M/n^p$.  Conversely, $A$ will always refer to
  the original (potentially full rank) matrix with block components
  $A=\left[A^1|A^2|\cdots|A^M\right]$, where $M = n^q$.  Furthermore,
  $A^{p,i}$ will always refer to the error free block of the original
  matrix $A$ whose entries correspond to the entries included in
  $B^{p,i}$.
  \footnote{That is, $B^{p,i}$ is used to approximate the singular
    values and left singular vectors of $A^{p,i}$ for all $p =
    1,\dots, q+1$, and $i = 1, \dots, M/n^{p-1}$} Thus, $A =
  \left[A^{p,1}|A^{p,2}|\cdots|A^{p,M/n^{(p-1)}}\right]$ holds for all
  $p = 1, \dots, q+1$, where $$A^{p+1,i} := \left[ A^{p,(i-1)n+1}
    ~\Big| \cdots \Big|~ A^{p,in} \right]$$ for all $p = 1,\dots, q$,
  and $i = 1, \dots, M/n^p$.  For $p=1$ we have $B^{1,i} = A^i =
  A^{1,i}$ for $i = 1, \dots, M$ by definition as per
  Algorithm~\ref{alg:parallelSVD}.  Our ultimate goal is to bound the
  renamed $\overline{\left(B^{q+1,1} \right)_d}$ matrix from lines 5
  and 6 of Algorithm~\ref{alg:parallelSVD} with respect to the
  original matrix $A$.  We will do this by induction on the level $p$.
  More specifically, we will prove that
\begin{enumerate}
\item $\overline{\left( B^{p,i} \right)_d} = A^{p,i} W^{p,i} + \Psi^{p,i}$, where 
\item $W^{p,i}$ is always a unitary matrix, and 
\item $\| \Psi^{p,i} \|_{\rm F} \leq \left( \left( 1 + \sqrt{2} \right)^{p} - 1 \right) \left \| (A^{p,i})_d - A^{p,i} \right \|_{\rm F}$,
\end{enumerate}
holds for all $p = 1,\dots, q+1$, and $i = 1, \dots, M/n^{(p-1)}$.

Note that conditions $1-3$ above are satisfied for $p=1$ since $B^{1,i} = A^i = A^{1,i}$ for all $i = 1, \dots, M$ by definition.  Thus, there exist unitary $W^{1,i}$ for all $i = 1, \dots, M$ such that
$$\overline{ \left( B^{1,i} \right)_d} ~=~ \overline{\left(  A^{1,i} \right)_d} ~=~ \left(  A^{1,i} \right)_d W^{1,i} ~=~ A^{1,i} W^{1,i} + \left(  \left(  A^{1,i} \right)_d - A^{1,i} \right) W^{1,i},$$
where $\Psi^{1,i} :=  \left(  \left(  A^{1,i} \right)_d - A^{1,i} \right) W^{1,i}$ has 
\begin{equation}
\| \Psi^{1,i} \|_{\rm F} = \left \| \left(  A^{1,i} \right)_d - A^{1,i} \right \|_{\rm F} \leq \sqrt{2} \left \| \left(  A^{1,i} \right)_d - A^{1,i} \right \|_{\rm F}.
\label{equ:StabModHere}
\end{equation}
Now suppose that conditions $1-3$ hold for some $p \in \{ 1, \dots, q\}$.  Then, one can see from condition 1 that
\begin{align*}
B^{p+1,i} ~&:=~ \left[ \overline{\left( B^{p,(i-1)n+1} \right)_d} ~\Big| \cdots \Big|~ \overline{\left( B^{p,in} \right)_d} \right]\\
&=~\left[ A^{p,(i-1)n+1} W^{p,(i-1)n+1} + \Psi^{p,(i-1)n+1}  ~\Big| \cdots \Big|~ A^{p,in} W^{p,in} + \Psi^{p,in}  \right]\\
&=~\left[ A^{p,(i-1)n+1} W^{p,(i-1)n+1} ~\Big| \cdots \Big|~ A^{p,in} W^{p,in} \right] + \left[  \Psi^{p,(i-1)n+1}  ~\Big| \cdots \Big|~ \Psi^{p,in}  \right] \\
&=~\left[ A^{p,(i-1)n+1} ~\Big| \cdots \Big|~ A^{p,in}  \right]  \tilde{W} + \tilde{\Psi},
\end{align*}
where $\tilde{\Psi} := \left[  \Psi^{p,(i-1)n+1}  ~\Big| \cdots \Big|~ \Psi^{p,in}  \right] $, and 
\[ 
  \arraycolsep=1.4pt\def\arraystretch{0.3}
  \! \! \tilde{W} := \! \! \! = \! \! \!
    \left(  \! \! 
    \begin{array}{c;{2pt/2pt}c;{2pt/2pt}c;{2pt/2pt}c}
        & & & \\
        & & & \\
        ~W^{p,(i-1)n+1} & {\bf 0} & {\bf 0} &  {\bf 0} \\ 
        \vspace{0.05in} \\ \hdashline[2pt/2pt] 
        {\bf 0} & W^{p,(i-1)n+2}  & {\bf 0} &  {\bf 0} \\ 
        \vspace{0.05in} \\ \hdashline[2pt/2pt]
        {\bf 0} & {\bf 0} & ~\ddots~ &  {\bf 0} \\ 
        \vspace{0.05in} \\ \hdashline[2pt/2pt]
        {\bf 0} & {\bf 0} & {\bf 0} &  W^{p,in}~
            \end{array}  \! \!   \right) \! \!.    \] 
Note that $\tilde{W}$ is unitary since its diagonal blocks are all unitary by condition 2.  Therefore, we have $B^{p+1,i} = A^{p+1,i} \tilde{W} + \tilde{\Psi}.$  

We may now bound $\left \| \left( B^{p+1,i} \right)_d - A^{p+1,i} \tilde{W}\right \|_{\rm F}$ using a similar argument to that employed in the proof of Theorem~\ref{thm:LowrankMerge}.
\begin{align}
\left \| \left( B^{p+1,i} \right)_d - A^{p+1,i} \tilde{W}\right \|_{\rm F} ~&\leq~ \left\| \left( B^{p+1,i} \right)_d - B^{p+1,i} \right\|_{\rm F} + \left\| B^{p+1,i}  - A^{p+1,i} \tilde{W} \right\|_{\rm F} \nonumber\\
~&=~ \sqrt{\sum^D_{j = d+1} \sigma^2_j \left(A^{p+1,i} \tilde{W} + \tilde{\Psi} \right)} ~+~ \| \tilde{\Psi} \|_{\rm F} \nonumber \\
~&\leq~ \sqrt{\sum^D_{j = d+1} 2 \sigma^2_j \left(A^{p+1,i} \tilde{W} \right) } ~+~ \sqrt{\sum^D_{j = 1} 2 \sigma^2_j( \tilde{\Psi}) }  + \| \tilde{\Psi} \|_{\rm F} \nonumber \\
~&=~ \sqrt{2} \left \| A^{p+1,i} - \left( A^{p+1,i} \right)_d \right \|_{\rm F} + \left(1+ \sqrt{2} \right) \| \tilde{\Psi} \|_{\rm F} \label{equ:LastthmRef1}.
\end{align}
Appealing to condition 3 in order to bound  $\| \tilde{\Psi} \|_{\rm F}$ we obtain
\begin{align*}
\| \tilde{\Psi} \|^2_{\rm F} ~=~ \sum^{n}_{j=1} \| \Psi^{p,(i-1)n+j} \|^2_{\rm F} &\leq \left( \left( 1 + \sqrt{2} \right)^{p} - 1 \right)^2 \sum^{n}_{j=1} \left \| (A^{p,(i-1)n+j})_d - A^{p,(i-1)n+j} \right \|^2_{\rm F}\\
&\leq \left( \left( 1 + \sqrt{2} \right)^{p} - 1 \right)^2 \sum^{n}_{j=1} \left \| (A^{p+1,i})^j_d - A^{p,(i-1)n+j} \right \|^2_{\rm F},
\end{align*}
where $(A^{p+1,i})^j_d$ denotes the block of $( A^{p+1,i} )_d$ corresponding to $A^{p,(i-1)n+j}$ for $j = 1, \dots, n$.  Thus, we have that
\begin{align}
\| \tilde{\Psi} \|^2_{\rm F} &\leq \left( \left( 1 + \sqrt{2} \right)^{p} - 1 \right)^2 \sum^{n}_{j=1} \left \| (A^{p+1,i})^j_d - A^{p,(i-1)n+j} \right \|^2_{\rm F} \nonumber \\ 
&= \left( \left( 1 + \sqrt{2} \right)^{p} - 1 \right)^2 \left \| (A^{p+1,i})_d - A^{p+1,i} \right \|^2_{\rm F} \label{equ:LastthmRef2}.
\end{align}
Combining \eqref{equ:LastthmRef1} and \eqref{equ:LastthmRef2} we can finally see that
\begin{align}
\left \| \left( B^{p+1,i} \right)_d - A^{p+1,i} \tilde{W}\right \|_{\rm F} &\leq \left[ \sqrt{2} + (1+ \sqrt{2} ) \left( \left( 1 + \sqrt{2} \right)^{p} - 1 \right) \right] \left \| \left( A^{p+1,i} \right)_d - A^{p+1,i}  \right \|_{\rm F} \nonumber\\
&= \left( \left( 1 + \sqrt{2} \right)^{p+1} - 1 \right) \left \| \left( A^{p+1,i} \right)_d - A^{p+1,i} \right \|_{\rm F} \label{equ:LastthmRef3}.
\end{align}
Note that $\left \| \left( B^{p+1,i} \right)_d - A^{p+1,i} \tilde{W}\right \|_{\rm F} = \left \| \overline{ \left( B^{p+1,i} \right)_d } - A^{p+1,i} W^{p+1,i}  \right \|_{\rm F}$ where $W^{p+1,i}$ is unitary.  Hence, we can see that conditions 1 - 3 hold for $p+1$ with $\Psi^{p+1,i} := \overline{ \left( B^{p+1,i} \right)_d } - A^{p+1,i} W^{p+1,i}$.
\end{proof}

Theorem~\ref{thm:LowRankRecovery} proves that Algorithm~\ref{alg:parallelSVD} will accurately compute low rank approximations of $\overline{A}$ whenever $q$ is chosen to be relatively small.  Thus, Algorithm~\ref{alg:parallelSVD} provides a distributed and incremental method for rapidly and accurately computing any desired number of dominant singular values/left singular vectors of $A$.   Furthermore, it is worth mentioning that the proof of Theorem~\ref{thm:LowRankRecovery} can be modified in order to prove that the $q$-level hierarchical Algorithm~\ref{alg:parallelSVD} is also robust/stable with respect to additive contamination of the initial input matrix $A$ by, e.g., round-off errors.  This can be achieved most easily by noting that the inequality in \eqref{equ:StabModHere} will still hold if 
$A^i = A^{1,i}$ is contaminated with arbitrary additive errors $E^i \in \mathbb{C}^{D \times N_i}$ having $\| E^i \|_{\rm F} \leq \frac{\sqrt{2} - 1}{\sqrt{D-d}+1} \left \| \left(  A^{1,i} \right)_d - A^{1,i} \right \|_{\rm F}$ for all $i = 1, \dots, M$.  We summarize this modification in the next lemma.

\begin{lemma}
Let $i \in [M]$.  Suppose that $B^{1,i} = A^{1,i} + E^i = A^i + E^i$ holds for $B^{1,i}, A^{1,i}, A^i, E^i \in \mathbb{C}^{D \times N_i}$, where $\| E^i \|_{\rm F} \leq \frac{\sqrt{2} - 1}{\sqrt{D-d}+1} \left \| \left(  A^{1,i} \right)_d - A^{1,i} \right \|_{\rm F}$.  Then, there exists a unitary matrix $W^{1,i} \in \mathbb{C}^{N_i \times N_i}$ such that 
$$\left \| \overline{ \left(  B^{1,i} \right)_d } - A^{1,i} W^{1,i}  \right\|_{\rm F} \leq  \sqrt{2} \left \| \left(  A^{1,i} \right)_d - A^{1,i} \right \|_{\rm F}.$$
In particular, the inequality in \eqref{equ:StabModHere} will still hold with $\Psi^{1,i} :=  \left[  \left(  B^{1,i} \right)_d - A^{1,i} \right] W^{1,i}$.
\label{lem:stabmodthm}
\end{lemma}

\begin{proof}
We have that $\left(  B^{1,i} \right)_d = A^{1,i} + \left[ \left(  B^{1,i} \right)_d - A^{1,i} \right]$ such that there exists unitary matrix $W^{1,i}$ with 
$$\overline{ \left( B^{1,i} \right)_d} ~=~ A^{1,i} W^{1,i} + \left[ \left(  B^{1,i} \right)_d - A^{1,i} \right] W^{1,i}.$$
As a result we have that
\begin{align*}
\left \| \overline{ \left(  B^{1,i} \right)_d } - A^{1,i} W^{1,i}  \right\|_{\rm F} &= \left \| \left(  B^{1,i} \right)_d - A^{1,i} \right\|_{\rm F}\\
&= \left \| \left(  A^{1,i} + E^i \right)_d - A^{1,i} \right\|_{\rm F}\\
&\leq \left \| \left(  A^{1,i} + E^i \right)_d - (A^{1,i} + E^i) \right\|_{\rm F} +  \left \| E^i  \right\|_{\rm F}\\
&= \sqrt{\sum^D_{j=d+1} \sigma^2_j\left( A^{1,i} + E^i \right)} +  \left \| E^i  \right\|_{\rm F}\\
&\leq \sqrt{\sum^D_{j=d+1} \left( \sigma_j \left( A^{1,i} \right) + \sigma_1\left( E^i \right) \right)^2 } +  \left \| E^i  \right\|_{\rm F}\\
&\leq  \left \| \left(  A^{1,i} \right)_d - A^{1,i}  \right\|_{\rm F} +  (1+\sqrt{D-d}) \left \| E^i  \right\|_{\rm F},
\end{align*}
where the last two inequalities utilize Weyl's inequality and the triangle inequality, respectively.  The desired result now follows from our assumed bound on $\| E^i \|_{\rm F}$.
\end{proof}

Combining Lemma~\ref{lem:stabmodthm} with the proof of Theorem~\ref{thm:LowRankRecovery} allows one to see that the statement of Theorem~\ref{thm:LowRankRecovery} will continue to hold even when $A$ is contaminated with small additive errors.  Having proven that Algorithm~\ref{alg:parallelSVD} is accurate, we are now free to consider its computational costs.

\subsection{Parallel Cost Model and Collectives}
To analyze the parallel communication cost of the hierarchical SVD
algorithm, the $\alpha$ -- $\beta$ -- $\gamma$ model for
distributed--memory parallel computation \cite{CPE1206} is used.  The
parameters $\alpha$ and $\beta$ respectively represent the latency
cost and the transmission cost of sending a floating point number
between two processors.  The parameter $\gamma$ represents the time
for one floating point operation (FLOP).

The $q$-level hierarchical Algorithm~\ref{algorithm:parallel_svd} seeks
to find the $d$ largest singular values and left singular vectors of a
matrix $A$.  If the matrix $A$ is decomposed into $M = n^q$ blocks,
where $n$ being the number of local SVDs being concatenated at each level,
the send/receive communication cost for the algorithm is 
\begin{align*}
  q\left(\alpha  + d\,(n-1)\,D\beta\right),
\end{align*}
assuming that the data is already distributed on the compute nodes and
no scatter command is required.  If the (distributed) right singular
vectors are needed, then a broadcast of the left singular vectors to
all nodes incurs a communication cost of $\alpha + d\, M \, \beta$.

Suppose $A$ is a $D\times N$ matrix, $N\gg D$.  The sequential SVD is
typically performed in two phases: bidiagonalization (which requires
$(2N\,D^2 + 2 D^3)$ flops) followed by diagonalization (negligible
cost).  If $M$ processing cores are available to compute the $q$-level
hierarchical SVD method in Algorithm~\ref{algorithm:parallel_svd}, and
the matrix $A$ is decomposed into $M=n^q$ blocks, where $n$ is again
the number of local SVDs being concatenated at each level.  The potential
parallel speedup can be approximated by
\begin{align}
  \begin{cases}
    \displaystyle
  \frac{(2 ND^2 + 2 D^3)\gamma} {\gamma\left[(2 (N/M)D^2 + 2D^3) + q(2dnD^2
      + 2D^3)\right] +q (\alpha + d(n-1)D\beta)}.
  & \text{if } nd > D\\
  \displaystyle
  \frac{(2 ND^2 + 2 D^3)\gamma} {\gamma\left[(2 (N/M)D^2 + 2D^3) + q(2(dn)^2D
      + 2(dn)^3)\right] +q (\alpha + d(n-1)D\beta)}.
  & \text{if } nd < D
  \end{cases}
  \label{eqn:weak_scaling}
\end{align}




\section{Numerical Experiments}
\label{sec:numerical_results}
The numerical experiments are organized into two categories.  First,
Theorems \ref{thm:FullRankRecovery} and \ref{thm:LowRankRecovery} are
validated -- the numerical experiments will demonstrate that the
incremental algorithms can recover full rank and low rank
approximations of $A$, up to rounding errors.  Then, numerical
evidence is presented to show the weak and strong scaling behavior of
the algorithms.

\subsection{Accuracy}
In this first experiment, we verify that singular values and the left
singular vectors of $A$ can be recovered using the incremental SVD
algorithm.  A full rank matrix of size $400\times128,000$ is
constructed in MATLAB, and its SVD computed to form the reference
solution.  The numerical experiment consists of partitioning the full
matrix into various block configurations before applying the
incremental SVD algorithm and comparing the singular values and left
singular vectors to the reference solution.  The scenarios and results
are reported in Table~\ref{tbl:full_rank_recovery}: $e_\sigma$ refers
to the maximum relative error of the singular values, i.e.,
\begin{align*}
  \max_{1\le i\le D} \frac{|\tilde{\sigma}_i - \sigma_i|}{|\sigma_i|},
\end{align*}
where $\sigma_i$ is the true singular value and $\tilde{\sigma}_i$ is
the singular value obtained using the hierarchical algorithm.  Similarly, $e_v$
refers to the maximum relative error of all the left singular vectors, i.e.,
\begin{align*}
  \max_{1\le i\le D} \frac{\|\tilde{v}_i - v_i\|_2}{\|v_i\|_2} ~=~\max_{1\le i\le D} \|\tilde{v}_i - v_i\|_2,
\end{align*}
where $v_i$ is the true singular value and $\tilde{v}_i$ is
the singular vector obtained using the hierarchical algorithm.
The number of blocks is $n^q$, where $q$ is the number of levels, and
$n$ is the number of sketches to merge at each level.  The incremental
algorithm recovers the singular values and left singular vectors of a
full rank matrix $A$, up to round-off errors.
\begin{table}[htbp]
  \centering
  \begin{tabular}{|c|c|c|c|c|c|}
    \hline
    $n$ & levels & \# blocks & block size & $e_\sigma$ & $e_v$ \\
    \hline
    2 & 1 & 2 & $400\times 64,000$ & $2.4\times10^{-13}$ & $2.3\times10^{-12}$ \\
     & 2 & 4 & $400\times 32,000$ & $1.4\times10^{-13}$  & $1.1\times10^{-12}$ \\
     & 3 & 8 & $400\times 16,000$ & $6.1\times10^{-14}$  & $2.2\times10^{-12}$ \\
     & 4 & 16 & $400\times 8,000$ & $5.3\times10^{-14}$  & $4.3\times10^{-12}$ \\
     & 5 & 32 & $400\times 4,000$ & $6.4\times10^{-14}$  & $4.3\times10^{-12}$ \\
     & 6 & 64 & $400\times 2,000$ & $5.1\times10^{-14}$  & $1.1\times10^{-12}$ \\
     & 7 & 128 & $400\times 1,000$ & $1.5\times10^{-13}$  & $1.5\times10^{-12}$ \\
     & 8 & 256 & $400\times 500$ & $1.6\times10^{-13}$  & $4.8\times10^{-12}$ \\
    \hline \hline
    4 & 1 & 4 & $400\times 16000$ & $2.3\times10^{-14}$  & $3.0\times10^{-12}$ \\
     & 2 & 16 & $400\times 1000$ & $2.3\times10^{-14}$  & $2.0\times10^{-12}$ \\
     & 3 & 256 & $400\times 125$ & $1.2\times10^{-14}$  & $2.5\times10^{-12}$ \\
    \hline
  \end{tabular}
  \caption{The number of blocks is $n^q$, where $q$ is the number of
    levels, and $n$ is the number of sketches to merge at each level.
    The maximum relative error of the singular values $\sigma$ and
    left singular vectors $v$ are reported.  The incremental algorithm
    recovers the singular values and left singular vectors of a full
    rank matrix $A$, up to round-off errors.}
  \label{tbl:full_rank_recovery}
\end{table}

In the second experiment, the incremental SVD algorithm is used to
recover a rank $d$ approximation, $(A)_d$, of a matrix, $A$.  The
matrix $A$ is again of size $400\times128,000$, and is formed by
taking the product $U\Sigma V'$, where $U$ and $V$ are random
orthogonal matrices of the appropriate size. The singular values are
chosen so that $\|(A)_d - A\|_F^2$ can be controlled.  The scenarios
and results are reported in Table~\ref{tbl:low_rank_recovery}.  Again,
the hierarchical algorithm recovers largest $d$  singular values
and associated left singular vectors to high accuracy.
\begin{table}[htbp]
  \centering
  \begin{tabular}{|c|c|c|c|c|c|c|}
    \hline
    $\|A-(A)_d\|_F^2$ & $n$ & levels & \# blocks & block size & $e_\sigma$ & $e_v$ \\
    \hline
    0.1 & 2 & 1 & 2 & $400\times 64,000$ & $2.3\times 10^{-13}$ & $8.3\times 10^{-9}$\\ 
     & & 2 & 4 & $400\times 32,000$ & $1.5\times 10^{-12}$ & $2.1\times 10^{-8}$ \\ 
     & & 3 & 8 & $400\times 16,000$ & $1.0\times 10^{-11}$ & $5.5\times 10^{-8}$ \\ 
     & & 4 & 16 & $400\times 8,000$ & $3.7\times 10^{-11}$ & $1.1\times 10^{-7}$ \\ 
     & & 5 & 32 & $400\times 4,000$ & $1.4\times 10^{-10}$ & $2.0\times 10^{-7}$ \\ 
     & & 6 & 64 & $400\times 2,000$ & $3.8\times 10^{-10}$ & $3.3\times 10^{-7}$ \\ 
     & & 7 & 128 & $400\times 1,000$ & $2.7\times 10^{-9}$ & $7.9\times 10^{-7}$ \\  
     & & 8 & 256 & $400\times 500$ & $9.9\times 10^{-9}$ & $1.3\times 10^{-6}$ \\  
    \cline{2-7}
    & 4 & 1 & 4 & $400\times 16000$ & $1.5\times 10^{-12}$ & $2.1\times 10^{-8}$ \\ 
     & & 2 & 16 & $400\times 1000$ & $3.7\times 10^{-11}$ & $1.3\times 10^{-7}$ \\  
     & & 3 & 256 & $400\times 125$ & $3.7\times 10^{-10}$ & $3.2\times 10^{-7}$ \\  
    \hline
    0.01 & 2 & 1 & 2 & $400\times 64,000$ & $2.1\times 10^{-14}$ & $8.2\times 10^{-12}$\\ 
     & & 2 & 4 & $400\times 32,000$ & $8.9\times 10^{-15}$ & $2.1\times 10^{-11}$ \\ 
     & & 3 & 8 & $400\times 16,000$ & $5.7\times 10^{-15}$ & $5.5\times 10^{-11}$ \\ 
     & & 4 & 16 & $400\times 8,000$ & $7.4\times 10^{-15}$ & $1.0\times 10^{-10}$ \\ 
     & & 5 & 32 & $400\times 4,000$ & $1.6\times 10^{-14}$ & $2.5\times 10^{-10}$ \\ 
     & & 6 & 64 & $400\times 2,000$ & $3.7\times 10^{-14}$ & $3.2\times 10^{-10}$ \\ 
     & & 7 & 128 & $400\times 1,000$ & $2.8\times 10^{-13}$ & $7.8\times 10^{-10}$ \\  
     & & 8 & 256 & $400\times 500$ & $9.6\times 10^{-13}$ & $1.2\times 10^{-9}$ \\  
    \cline{2-7}
    & 4 & 1 & 4 & $400\times 32,000$ & $1.7\times 10^{-14}$ & $2.1\times 10^{-11}$ \\ 
     & & 2 & 16 & $400\times 8,000$ & $1.2\times 10^{-14}$ & $1.0\times 10^{-10}$ \\  
     & & 3 & 256 & $400\times 500$ & $1.4\times 10^{-14}$ & $3.1\times 10^{-10}$ \\  
    \hline
  \end{tabular}
  \caption{Low rank approximation of $A$ computed using the
    incremental SVD algorithm, compared against the true low rank
    approximation of $A$.}
  \label{tbl:low_rank_recovery}
\end{table}

\subsection{Parallel Scaling}

For the first scaling experiment, the left singular vectors and the
singular values of a matrix $A$ ($D = d =800$, $N = 1,152,000$) are
found using an MPI implementation of our hierarchical algorithm, and a
threaded LAPACK SVD algorithm, {\tt dgesvd}, implemented in the
threaded Intel MKL library.  Computational nodes from the stampede
cluster at the Texas Advanced Computing Center (TACC) were used for
the parallel scaling experiments.  Each node is outfitted with 32GB of
RAM and two Intel Xeon E5-2680 processors for a total of 16 processing
cores per node.  In a pre-processing step, the matrix $A$ is
decomposed, with each block of $A$ stored in separate HDF5 files; the
generated HDF5 files are hosted on the high-speed Lustre server.  The
observed speedup is summarized in
Figure~\ref{fig:scaling_study} \footnote{The raw data used to generate
  the speedup curves for Figure~\ref{fig:scaling_study} is presented
  in Tables~\ref{tbl:mkl_strong_scaling} and
  \ref{tbl:hierarchical_strong_scaling} in the Appendix.}.  In the blue
curve, the observed speedup is reported for a varying number of MKL
worker threads.  In the red curve, the speedup is reported for a
varying number of worker threads $i$, applied to an appropriate
decomposition of the matrix.  Each worker uses the same Intel MKL
library to compute the SVD of the decomposed matrices (each using a
single thread), the proxy matrix is assembled, and the master thread
computes the SVD of the proxy matrix using the Intel MKL library,
again with a single thread.  The parallel performance of our
distributed SVD is superior to the threaded MKL library, this in spite
of the fact that our algorithm was implemented using MPI 2.0 and does
not leverage the inter-node communication savings that is possible
with newer MPI implementations.
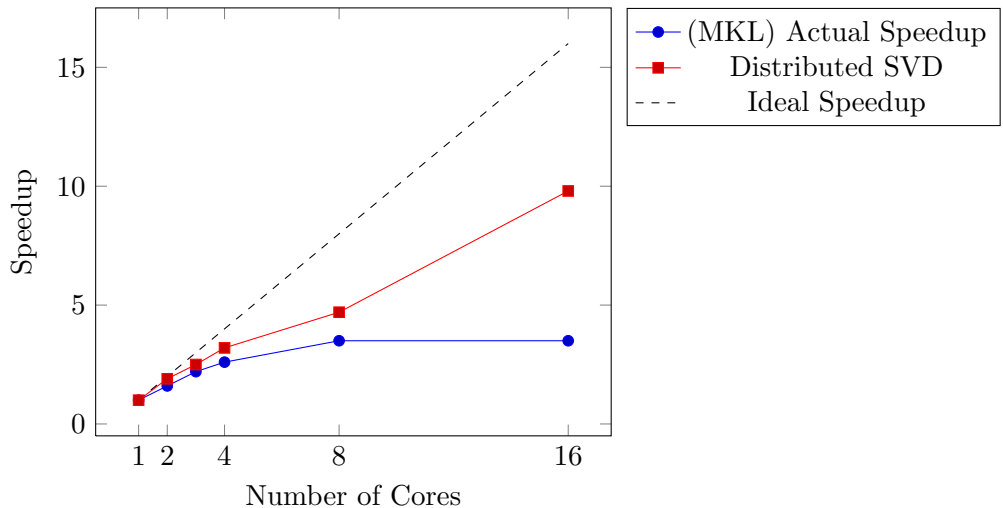
\begin{figure}[htbp]
  \centering
  \begin{tikzpicture}
    \begin{axis}[
        xlabel={Number of Cores},
        xtick={1,2,4,8,16},
        yminorticks=true,
        legend pos = outer north east,
        ylabel={Speedup},
      ]
      
      \addplot coordinates {
        (1, 1)
        (2, 1.6)
        (3, 2.2)
        (4, 2.6)
        (8, 3.5)
        (16, 3.5)
      };
      \addlegendentry{(MKL) Actual Speedup}

      \addplot coordinates {
        (1, 1)
        (2, 1.9)
        (3, 2.5)
        (4, 3.2)
        (8, 4.7)
        (16, 9.8)
      };
      \addlegendentry{Distributed SVD}
      
      \addplot [color=black,dashed] coordinates {
        (1, 1)
        (2, 2)
        (4, 4)
        (8, 8)
        (16, 16)
      };
      \addlegendentry{Ideal Speedup}
      \end{axis}
  \end{tikzpicture}
  \caption{Strong scaling study of the {\tt dgesvd} function in the
    threaded Intel MKL library (blue) and the proposed distributed SVD
    algorithm (red).  The input matrix is of size
    $800\times1,152,000$.  }
  \label{fig:scaling_study}
\end{figure}

The second parallel scaling experiment is a weak scaling study, where
the size of the input matrix $A$ is varied depending on the number of
computing cores, $A = 2000\times 32,000M$, where $M$ is the number of
compute cores.  The theoretical peak efficiency is computed using
equation~\ref{eqn:weak_scaling}, assuming negligible communication
overhead (i.e., we set $\alpha=\beta=0$).  There is slightly better
efficiency if one merges for larger $n$ (more sketches are merged
together at each level).  This behavior will persist until the size of
the proxy matrix is comparable or larger than the original blocksize.
The asymptotic behavior of the efficiency curves agree with the
theoretical peak efficiency curve if $M>16$.
\begin{figure}[htbp]
  \centering
  \begin{tikzpicture}
    \begin{semilogxaxis}[
        xlabel={Number of Cores},
        yminorticks=true,
        legend pos = outer north east,
        ylabel={Efficiency},
        ymin=0, ymax=1,
        log basis x={2}]]
      
      \addplot[mark=oplus*, color=blue] coordinates {
        (1, 1)
        (2, 0.78)
        (4, 0.55)
        (8, 0.38)
        (16, 0.32)
        (32, 0.26)
        (64, 0.23)
        (128, 0.19)
        (256, 0.16)
        (512, 0.15)
      };
      \addlegendentry{$n=2$}

      \addplot[mark=none,style=dashed,color=blue] coordinates {

        (1,1)
        (2,1.65/2)
        (4,2.83/4)
        (8,4.96/8)
        (16, 8.86/16)
        (32, 16/32)
        (64, 29.3/64)
        (128,53.9/128)
        (256,99.9/256)
        (512,186/512)
      };
      \addlegendentry{theoretical peak, $n=2$}
      
      \addplot [mark=square*,color=red] coordinates {
        (1, 1)
        (3, 0.7)
        (9, 0.43)
        (27, 0.34)
        (81, 0.28)
        (243, 0.22)
      };
      \addlegendentry{$n=3$}

      \addplot [mark=none,style=dashed,color=red] coordinates {
        (1, 1)
        (3, 2.33/3)
        (9, 5.8/9)
        (27, 14.94/27)
        (81, 39.3/81)
        (243, 105/243)
      };
      \addlegendentry{theoretical peak, $n=3$}
      
      \addplot [mark=diamond*,color=brown] coordinates {
        (1, 1)
        (4, 0.65)
        (16, 0.42)
        (64, 0.32)
        (256, 0.24)
      };

      \addlegendentry{$n=4$}

      \addplot [mark=none,color=brown, style=dashed] coordinates {
        (1, 1)
        (4, 2.95/4)
        (16, 9.51/16)
        (64, 32.0/64)
        (256, 110.7/256)
      };

      \addlegendentry{theoretical peak, $n=4$}

    \end{semilogxaxis}
  \end{tikzpicture}
  \caption{Weak scaling study of the one-level SVD algorithm.  The
    input matrix is of size $2000\times (32000 M)$, where $M$ is the
    processing cores used in the computation.  The observed efficiency
    is plotted for various $n$'s (number of scaled singular vectors
    concatenated at each hierarchical level).  There is a slight
    efficiency gain when increasing $n$.}
  \label{fig:weak_scaling}
\end{figure}
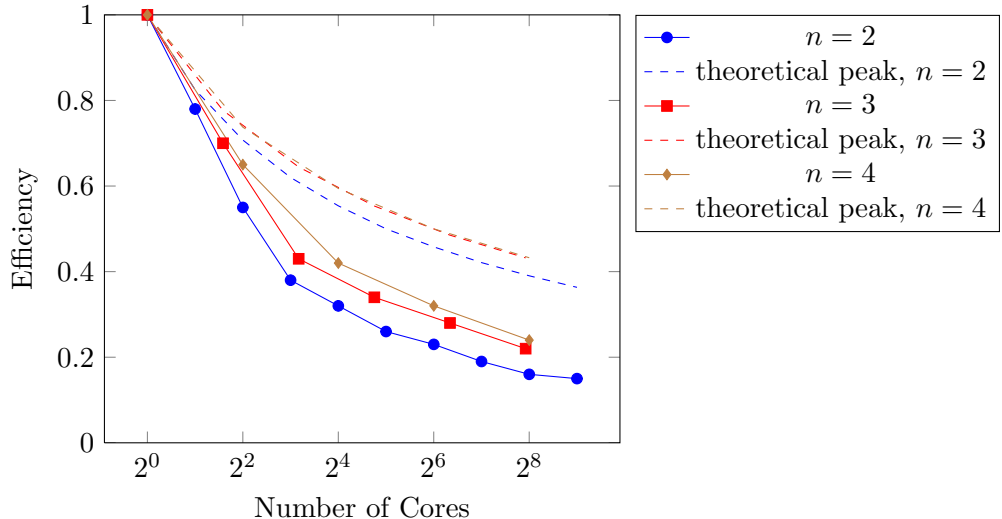

In the last experiment, we repeat the weak scaling study (where the
size of the input matrix $A$ is varied depending on the number of
worker nodes, $A = 2000\times 32,000M$, where $M$ is the number of
compute cores, but utilize a priori knowledge that the rank of $A$ is
much less than the ambient dimension.  Specifically, we construct a
data set with $d = 200 \ll 2000$.  The hierarchical SVD performs more
efficiently if the intrinsic dimension of the data can be estimated a
priori.  Similar observations can be made to the previous experiment:
the asymptotic behavior of the efficiency curves agree with the
theoretical peak efficiency curve if $M>16$.
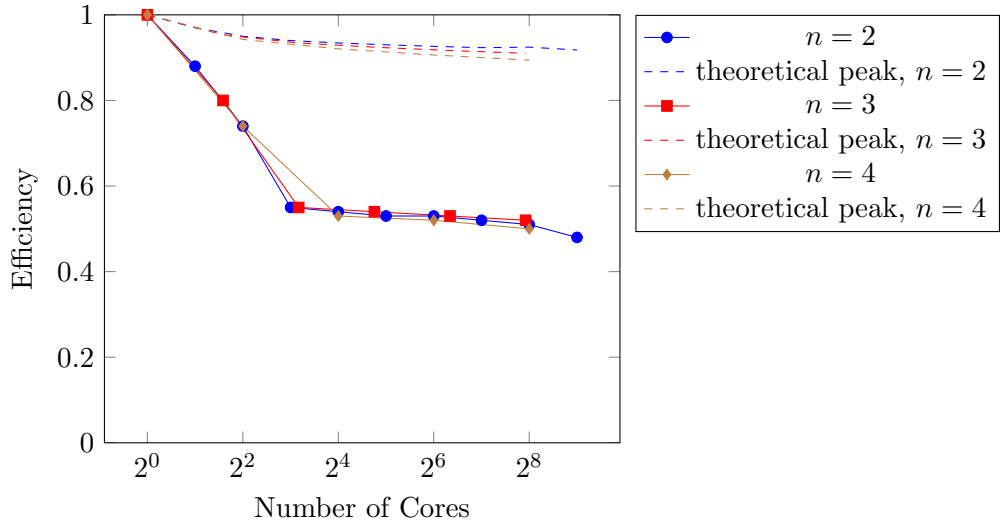
\begin{figure}[htbp]
  \centering
  \begin{tikzpicture}
    \begin{semilogxaxis}[
        xlabel={Number of Cores},
        yminorticks=true,
        legend pos = outer north east,
        ylabel={Efficiency},
        ymin=0,
        ymax=1,
        log basis x={2}]]

      \addplot[mark=oplus*, color=blue] coordinates {
        (1, 1)
        (2, 0.88) 
        (4, 0.74)
        (8, 0.55)
        (16, 0.54)
        (32, 0.53)
        (64, 0.53)
        (128, 0.52)
        (256, 0.51)
        (512, 0.48)
      };
      \addlegendentry{$n=2$}

      \addplot[mark=none,style=dashed,color=blue] coordinates {
        (1,1)
        (2,1.94/2)
        (4,3.80/4)
        (8,7.52/8)
        (16, 14.95/16)
        (32,29.76/32)
        (64, 59.29/64)
        (128,118.19/128)
        (256,236.68/256)
        (512,470.00/512)
      };
      \addlegendentry{theoretical peak, $n=2$}

      \addplot [mark=square*,color=red] coordinates {
        (1, 1)
        (3, 0.8)
        (9, 0.55)
        (27,0.54)
        (81, 0.53)
        (243, 0.52)
      };
      \addlegendentry{$n=3$}

      \addplot [mark=none,style=dashed,color=red] coordinates {
        (1, 1)
        (3, 2.86/3)
        (9, 8.41/9)
        (27, 24.96/27)
        (81, 74.25/81)
        (243, 221.15/243)
      };
      \addlegendentry{theoretical peak, $n=3$}      

      \addplot [mark=diamond*,color=brown] coordinates {
        (1, 1)
        (4, 0.74)
        (16, 0.53)
        (64, 0.52)
        (256, 0.50)
      };
      \addlegendentry{$n=4$}

      \addplot [mark=none,color=brown, style=dashed] coordinates {
        (1, 1)
        (4, 3.77/4)
        (16, 14.73/16)
        (64, 58.00/64)
        (256, 228.94/256)
      };
      \addlegendentry{theoretical peak, $n=4$}

    \end{semilogxaxis}
  \end{tikzpicture}
  \caption{Weak scaling study of the hierarchical SVD algorithm
    applied to data with intrinsic dimension much lower than the
    ambient dimension.  The input matrix is of size $2000\times (32000
    M)$, where $M$ is the processing cores used in the computation.
    The intrinsic dimension is $d=200 \ll 2000$.  The observed
    efficiency is plotted for various $n$'s (number of scaled singular
    vectors concatenated at each hierarchical level).  As expected, the
    theoretical and observed efficiency are better if the intrinsic
    dimension is known (or can be estimated) a priori.} 
  \label{fig:weak_scaling_rankd}
\end{figure}

\section{Concluding Remarks and Acknowledgments}

In this paper, we show that the SVD of a matrix can be constructed
efficiently in a hierarchical approach.  Our algorithm is proven to
recover exactly the singular values and left singular vectors if the
rank of the matrix $A$ is known.  Further, the hierarchical algorithm
can be used to recover the $d$ largest singular values and left
singular vectors with bounded error.  We also show that the proposed
method is stable with respect to roundoff errors or corruption of the
original matrix entries.  Numerical experiments validate the proposed
algorithms and parallel cost analysis.

The authors note that the practicality of the hierarchical algorithm
is questionable for sparse input matrices, since the assembled proxy
matrices as posed will be dense.  Further investigation in this
direction is required, but beyond the scope of this paper.  Lastly,
the hierarchical algorithm has a map--reduce flavor that will lend
itself well to a map reduce framework such as Apache Hadoop
\cite{White:2009:HDG:1717298} or Apache Spark
\cite{Zaharia:2010:SCC:1863103.1863113}.


\bibliographystyle{abbrv}
\bibliography{svd}

\appendix
\section{Raw data from numerical experiments}
This section contains raw data used to generate the speedup curves in
Section~\ref{sec:numerical_results} for the parallel scaling studies.
The simulations were conducted on the stampede cluster at the Texas
Advanced Computing Center (TACC). Each node is outfitted with 32GB of
RAM and two Intel Xeon E5-2680 processors for a total of 16 processing
cores per node.  The performance application programming interface
(PAPI) \cite{Browne01082000} was used to obtain an estimate of the
computational throughput obtained by the implementations.  

\begin{table}[htbp]
  \centering
  \begin{tabular}{|c|c|c|c|}
    \hline
    \# threads & Walltime (s)& Speedup & MFLOP/$\mu$s$^*$ \\
    \hline
    1  & 245 & -- & 40 \\
    2  & 151 & 1.6 & 68 \\
    3  & 110 & 2.2 & 100 \\
    4  &  91 & 2.6 & 126 \\
    8  &  71 & 3.5 & 194 \\
    16 &  71 & 3.5 & 224 \\
    \hline
  \end{tabular}
  \caption{Raw data used to compute the speedup curve for
    Figure~\ref{fig:scaling_study} using the threaded MKL {\tt dgesvd}
    implementation.  The reported MFLOP/$\mu$s is estimated by scaling
    the throughput of the master thread reported by PAPI.  PAPI is
    unable to report the total MFLOP/$\mu$s since the MKL library uses
    pthreads to spawn and remove threads as needed.}
  \label{tbl:mkl_strong_scaling}
\end{table}

\begin{table}[htbp]
  \centering
  \begin{tabular}{|c|c|c|c|c|c|c|}
    \hline
    $n$ & levels & \# blocks & block size & Walltime (s)& Speedup & MFLOP/$\mu$s \\
    \hline
    1  & 1 & 1  & $800\times 1,152,000$ & 245 & --  & 40\\
    2  & 1 & 2  & $800\times 768,000$   & 126 & 1.9 & 81  \\
    3  & 1 & 3  & $800\times 512,000$   & 97  & 2.5 & 105 \\
    4  & 1 & 4  & $800\times 384,000$   & 76  & 3.2 & 136 \\
    8  & 1 & 8  & $800\times 192,000$   & 52  & 4.7 & 204 \\
    16 & 1 & 16 & $800\times 96,000$    & 25  & 9.8 & 447 \\
    \hline
  \end{tabular}
  \caption{Raw data used to compute the speedup curve for
    Figure~\ref{fig:scaling_study} using the incremental SVD implementation.}
  \label{tbl:hierarchical_strong_scaling}
\end{table}

\begin{table}[htbp]
  \centering
  \begin{tabular}{|c|c|c|c|c|c|}
    \hline
    $n$ & levels & \# blocks & Walltime (s)& Efficiency & MFLOP/$\mu$s \\
    \hline
    2  & 0 & 1    & 32  & 1    & 25  \\
       & 1 & 2    & 41  & 0.78 & 46  \\
       & 2 & 4    & 58  & 0.55 & 74  \\
       & 3 & 8    & 84  & 0.38 & 110 \\
       & 4 & 16   & 101 & 0.32 & 185 \\
       & 5 & 32   & 121 & 0.26 & 314 \\
       & 6 & 64   & 138 & 0.23 & 552\\
       & 7 & 128  & 172 & 0.19 & 907\\
       & 8 & 256  & 195 & 0.16 & 1641\\
       & 9 & 512  & 218 & 0.15 & 2928 \\
    \hline
    3 & 0 & 1   & 32  & 1    & 25\\
      & 1 & 3   & 46  & 0.70 & 63 \\
      & 2 & 9   & 74  & 0.43 & 127 \\
      & 3 & 27  & 94  & 0.34 & 305 \\
      & 4 & 81  & 113 & 0.28 & 767 \\
      & 5 & 243 & 148 & 0.22 & 1819\\
    \hline
    4 & 0 & 1   & 32  & 1 & 25 \\
      & 1 & 4   & 49  & 0.65 & 78 \\
      & 2 & 16  & 77  & 0.42 & 209 \\
      & 3 & 64  & 99  & 0.32 & 660 \\
      & 4 & 256 & 131 & 0.24 & 2041 \\
    \hline
  \end{tabular}
  \caption{Raw data used to compute the efficiency curves for
    Figure~\ref{fig:weak_scaling} using the hierarchical SVD
    implementation. Each block is of size $2000\times32,000$.  In
    this experiment, all 2000 singular values and left singular vectors
    are computed. }
  \label{tbl:hierarchical_weak_scaling_full_rank}
\end{table}

\begin{table}[htbp]
  \centering
  \begin{tabular}{|c|c|c|c|c|c|}
    \hline
    $n$ & levels & \# blocks & Walltime (s)& Efficiency & MFLOP/$\mu$s \\
    \hline
    2  & 0 & 1    & 28 & 1 & 29 \\
       & 1 & 2    & 32 & 0.88 & 51 \\
       & 2 & 4    & 38 & 0.74 & 90 \\
       & 3 & 8    & 51 & 0.55 & 135 \\
       & 4 & 16   & 52 & 0.54 & 268 \\
       & 5 & 32   & 53 & 0.53 & 528 \\
       & 6 & 64   & 53 & 0.53 & 1041 \\
       & 7 & 128  & 54 & 0.52 & 2056 \\
       & 8 & 256  & 55 & 0.51 & 4065 \\
       & 9 & 512  & 58 & 0.48 & 7889 \\
    \hline
    3 & 0 & 1   & 28 & 1 & 29 \\
      & 1 & 3   & 35 & 0.80 & 72 \\
      & 2 & 9   & 51 & 0.55 & 151 \\
      & 3 & 27  & 52 & 0.54 & 445 \\
      & 4 & 81  & 53 & 0.53 & 1322 \\
      & 5 & 243 & 54 & 0.52 & 3905 \\
    \hline
    4 & 0 & 1   & 28 & 1 & 29 \\
      & 1 & 4   & 38 & 0.74 & 89 \\
      & 2 & 16  & 53 & 0.53 & 264 \\
      & 3 & 64  & 54 & 0.52  & 1021 \\
      & 4 & 256 & 56 & 0.50 & 3940 \\
    \hline
  \end{tabular}
  \caption{Raw data used to compute the efficiency curves for
    Figure~\ref{fig:weak_scaling_rankd} using the hierarchical SVD
    implementation.  Each block is of size $2000\times32,000$.  In
    this experiment, the input matrix is of rank $d=200$; the
    hierarchical algorithm finds the 200 singular values and left
    singular vectors.}
  \label{tbl:hierarchical_weak_scaling_rankd}
\end{table}

\end{document}